\documentclass[11pt, reqno]{amsart}

\usepackage{xander_style}
\usepackage{xander_macros}

\title{On the field of definition of a cubic rational function and its critical points}
\author{Xander Faber}
\address{Center for Computing Sciences \\
Institute for Defense Analyses \\
Bowie, MD} 
\email{awfaber@super.org}

\author{Bianca Thompson}
\address{Westminster College \\
Salt Lake City, UT}
\email{bthompson@westminstercollege.edu}


\begin{document}
\maketitle
\begin{abstract}
Using essentially only algebra, we give a proof that a cubic rational function over $\CC$ 
with real critical points is equivalent to a real rational function. We also 
show that the natural generalization to $\QQ_p$ fails for all $p$.
% \textit{2010 Mathematics Subject Classification.} ?? (primary); ?? (secondary) 
\end{abstract}

%%%%%%%%%%%%%%%%%%%%%%%%%%%%%%%%%%%%%%%%%%%%%%%%%%%%%%%%%%%%%%%%%%%%%%%%%%%%%%%%
%%%%%%%%%%%%%%%%%%%%%%%%%%%%%%%%%%%%%%%%%%%%%%%%%%%%%%%%%%%%%%%%%%%%%%%%%%%%%%%%
        
\section{Introduction}
Let $K$ be a field of characteristic zero with algebraic closure $\bar K$. We 
say  that two rational functions $f,g \in \bar K(z)$ are \textbf{equivalent} if 
 there is a fractional linear transformation $\sigma \in \bar K(z)$ such that   
 $f = \sigma \circ g$. Viewing $f$ and $g$ as endomorphisms of the projective 
 line,  we see that they are equivalent if they differ by a change of 
 coordinate  on the target. Note that equivalent rational functions have the 
 same  critical points. This is the equivalence relation used in the study of
 \textit{dessins d'enfants}, as opposed to the equivalence used for dynamical systems. 

\begin{theorem}[Eremenko/Gabrielov]
If $f \in \CC(z)$ is a rational function with real critical points, then 
$f$  is equivalent to a rational function with real coefficients. 
\end{theorem}

By relating equivalence classes of rational functions with special Schubert 
cycles,  Goldberg \cite{Goldberg_Catalan_1991} showed that there are at most    
\[
\rho(d) := \frac{1}{d}\binom{2d-2}{d-1}
\] 
equivalence classes of degree~$d$ 
rational functions with a given set of critical points. Eremenko and Gabrielov 
\cite{EG_real_critical_points,EG_elementary_real_critical_points} used 
topological,  combinatorial, and complex analytic techniques to 
\textit{construct}  exactly $\rho(d)$ real rational functions with a given set 
of real  critical points, which proves the theorem. 

But the correspondence between a rational function and its critical points is   
purely algebraic, via roots of the derivative. This raises the question of 
whether  a truly elementary proof of the above result exists --- one that does 
not  use any analysis or topology. We give such a proof for cubic functions in 
this  note. 

For a field $K$ and a nonconstant rational function $\phi \in K(z)$, we say that
$K$ is \textbf{$\phi$-perfect} if the map $\phi \colon \PP^1(K) \to \PP^1(K)$ is
surjective. For example, if $K$ has characteristic~$p$ and $\phi(z) = z^p$, then
$K$ is $\phi$-perfect if and only if it is a perfect field in the usual sense.

\begin{theorem}
\label{Thm: Main theorem}
Let $K$ be a field of characteristic zero with algebraic closure $\bar K$. The 
following  statements are equivalent:
\begin{enumerate}
\item\label{Item: 1} Any cubic rational function $f \in \bar K(z)$ with 
$K$-rational  critical points is equivalent to a rational function in $K(z)$. 
\item\label{Item: 2} $K$ is $\phi$-perfect for the function 
\[\phi(z) = -\frac{z^2+2z}{2z+3}.\]
\item\label{Item: 3} $(2y-1)^2+3$ is a square in $K$ for every $y\in K.$ 
\end{enumerate}
\end{theorem}
%Referee fix: reference theorem here
Theorem~\ref{Thm: Main theorem} will be proved in \S\ref{Sec: main proof}. 

\begin{corollary}
If $f \in \CC(z)$ is a cubic rational function with real critical points,  
then $f$ is equivalent to a real rational function. 
\end{corollary}

\begin{proof}
Evidently $(2y-1)^2 + 3$ is a square in $\RR$ for every $y \in \RR$. 
\end{proof}

\begin{remark}
For a quadratic function $f$, the correspondence between the field of definition of $f$ and the field of definition of its critical points is trivial. Direct computation shows that a 
function  with critical points $c_1 \in \PP^1(K)$ and $c_2 \in \PP^1(K) \smallsetminus \{\infty\}$ is equivalent to
\[
\begin{cases}
(z-c_2)^2 & \text{ if } c_1 = \infty \\
\left(\frac{z-c_1}{z-c_2}\right)^2 & \text{ if } c_1 \not= \infty.
\end{cases}
\]
\end{remark}

For what other fields of interest $K$ does the corollary on equivalence of 
rational  functions continue to hold? Said another way, which fields $K$ are    
$\phi$-perfect for $\phi(z) = -\frac{z^2+2z}{2z+3}$? 

For one source of (non-)examples, we look at non-Archimedean completions of the rational numbers. 

\begin{proposition}
\label{Prop: p-adic}
Set $\phi(z) = -\frac{z^2+2z}{2z+3}$. The field $\QQ_p$ is not $\phi$-perfect
for any prime $p$.
\end{proposition}

This proposition will be proved in \S\ref{Sec: p-adic}. Using Theorem~\ref{Thm: 
Main theorem},  we obtain the following as a consequence: 

\begin{corollary} Let $p$ be a prime. 
There exists a cubic rational function $f \in \bar \QQ_p(z)$ that has
$\QQ_p$-rational critical points, but that is not equivalent to a rational
function with coefficients in $\QQ_p$.
\end{corollary}

\begin{remark}
  In fact, for fixed $p$ one can use the argument from the proof to show that
  there are infinitely many pairwise inequivalent cubic rational functions $f
  \in \bar \QQ_p(z)$ satisfying the statement of the corollary.
\end{remark}

% In \S\ref{Sec: number fields} we will show that number fields are not $\phi$-perfect for 
% any quadratic rational function $\phi$.

The authors would like to thank Sebastian Dorn for pointing out an error in
Proposition~\ref{Prop: p-adic} in the published version of this paper
\cite{Faber_Thompson_2016}. This issue was addressed in an official corrigendum
\cite{Faber_Thompson_corrigendum_2016}.

%%%%%%%%%%%%%%%%%%%%%%%%%%%%%%%%%%%%%%%%%%%%%%%%%%%%%%%%%%%%%%%%%%%%%%%%%%%%%%%%
%%%%%%%%%%%%%%%%%%%%%%%%%%%%%%%%%%%%%%%%%%%%%%%%%%%%%%%%%%%%%%%%%%%%%%%%%%%%%%%%

\section{Proof of the theorem}
\label{Sec: main proof}

We begin with a normal form for cubic functions. For $u \in \bar K
\smallsetminus \{-1,-2\}$, define
\begin{equation}
\label{Eq: Normal form}
f_u(z) = \frac{z^2(z+u)}{(2u+3)z - (u+2)}.
\end{equation}
(We exclude $u = -1,-2$ because otherwise a root of the numerator and the
denominator collide, and $f_u$ degenerates to a quadratic.) This function has
the property that it fixes $0,1$, and $\infty$, and each of these three points
is critical.

\begin{lemma}
A cubic rational function that is critical at $0$, $1$, and $\infty$ is 
equivalent  to a unique $f_u$, and the fourth critical point is 
$\phi(u) = - \frac{u^2+2u}{2u+3}$. 
\end{lemma}

\begin{proof}
Write $f$ for a cubic function that is critical at $0$, $1$, and $\infty$. By a 
change of coordinate on the target, we may assume that $0,1$, and $\infty$ are 
all  fixed points. (The critical values of a cubic are distinct for local degree reasons.) 
Thus, $f$ is of the form
\[
f(z) = \frac{z^3 + uz^2}{vz + (u - v+1)}
\]
for some $u,v \in \bar K$. 
The Wronskian has the form 
\[
z(2vz^2 + (uv + 3(u-v+1))z + 2u(u-v+1)).
\]
Substituting $z = 1$ kills this expression since $1$ is a critical point. 
Solving  the resulting equation for $v$ yields $v = 2u + 3$. Hence $f$ is 
equivalent  to \eqref{Eq: Normal form}, as desired.

For the uniqueness statement, suppose that $f_u$ is equivalent to $f_v$ for some 
$u,v$. Then there is a fractional linear $\sigma \in \bar K(z)$ such that 
$f_u = \sigma \circ f_v$. But $f_u$ and $f_v$ both fix $0,1$, and $\infty$, so 
that  $\sigma$ does as well. This means $\sigma(z) = z$, and $u = v$. 

The fourth critical point of $f_u$ may be found by factoring the derivative. 
\end{proof}

\begin{remark}
Note that taking $u = 0, -3$, or $-3/2$ gives a double critical point at $0,1$, 
or $\infty$,  respectively. 
\end{remark}

\begin{proposition}
If $f_u \in \bar K(z)$ is equivalent to a rational function with 
$K$-coefficients,  then $u \in K$. 
\end{proposition}

\begin{proof}
Let $\sigma \in \bar K(z)$ be a fractional linear map such that $\sigma \circ 
f_u$  has coefficients in $K$. The images of $0,1$, and $\infty$ under 
$\sigma \circ f_u$ all lie in $\PP^1(K)$. We may therefore apply a further 
fractional  linear transformation $\tau$ with $K$-coefficients so that $\tau 
\circ \sigma \circ f_u$  fixes $0,1$, and $\infty$. That is, $\tau \circ \sigma 
\circ f_u = f_v$  for some $v$. Since $\tau$ and $\sigma \circ f_u$ have 
$K$-coefficients, we know  that $v \in K$. By uniqueness in the lemma, we 
conclude that  $u = v$. 
\end{proof}

\begin{proof}[Proof of Theorem ~\ref{Thm: Main theorem}] To prove the implication 
\eqref{Item: 1} $\Rightarrow$\eqref{Item: 2}, we take $y \in K$ and attempt to  
solve the equation $\phi(u) = y$ with $u \in K$. If $y = \infty$, then we may   
take $u = -3/2$. Otherwise, choose $u \in \bar K$ such that $\phi(u) = y$. Then 
 the function $f_u$ has $K$-rational critical points $\{0,1,\infty, y\}$. By 
 \eqref{Item: 1}, $f_u$ is equivalent to a rational function with 
 $K$-coefficients. The  above proposition implies that $u \in K$. 

To prove \eqref{Item: 2} $\Rightarrow$ \eqref{Item: 1}, we start with a 
rational  function $f \in \bar K(z)$ with $K$-rational critical points. If 
$f$  has only two critical points, then each must have multiplicity~2 (by the 
Riemann-Hurwitz formula).  Without loss, we assume they are $0$ and $\infty$, 
and that $0,\infty$ are fixed by $f$, so  that $f(z) = az^2$ for some $a \in 
\bar K$. Evidently $a^{-1} f$ has coefficients  in $K$. 

Now suppose that $f$ has at least three distinct critical points. Without loss, we may   
assume that $0,1$, and $\infty$ are among them. In particular, by the lemma we  
 see that $f$ is equivalent to $f_u$ for some $u \in \bar K$. The remaining     
 critical point is $\phi(u) \in K$. By \eqref{Item: 2}, both solutions of $\phi(z) = 
 \phi(u)$  lie in $\PP^1(K)$, so that $u \in K$. That is, $f$ is equivalent to 
 a rational function with $K$-coefficients. 
 
 To prove \eqref{Item: 2} $\Leftrightarrow$ \eqref{Item: 3}, choose $y \in K$ and consider 
 \[
 \phi(z)=-\frac{z^2+2z}{2z+3}=y.
 \] 
Rearranging, we get a quadratic in $z$ with discriminant 
\[
(2y+2)^2- 4 \cdot 3y = (2y-1)^2 + 3.
 \]
Thus, we can solve for $z \in K$ if and only if $(2y-1)^2 + 3$ is a square in $K$. 
\end{proof}

%%%%%%%%%%%%%%%%%%%%%%%%%%%%%%%%%%%%%%%%%%%%%%%%%%%%%%%%%%%%%%%%%%%%%%%%%%%%%%%%
%%%%%%%%%%%%%%%%%%%%%%%%%%%%%%%%%%%%%%%%%%%%%%%%%%%%%%%%%%%%%%%%%%%%%%%%%%%%%%%%

\section{$p$-adic fields}
\label{Sec: p-adic}

Our proof of Proposition~\ref{Prop: p-adic} is split into the subcases $p = 2$,
$p = 3$, and $p > 3$. We want to show that $\QQ_p$ is not $\phi$-perfect for
$\phi(z){}={}-\frac{z^2+2z}{2z+3}$. This
amounts to determining whether $\phi(z) = y$ has a solution in $\PP^1(\QQ_p)$
for $y \in \QQ_p$.  Rearranging gives the quadratic equation $z^2 + 2(1+y)z + 3y
= 0$, which has discriminant
\begin{equation}
\label{Eq: discriminant}
\Delta = 4(y^2 - y + 1) = (2y - 1)^2 + 3.
\end{equation}
Determining if $\QQ_p$ is $\phi$-perfect now amounts to determining whether 
$\Delta$  is a square in $\QQ_p$ for every $y \in \QQ_p$. 

For $p = 2$, set $y = \frac{1}{2} + t$ with $t \in \ZZ_2$. Then \eqref{Eq: 
discriminant}  becomes 
\[
\Delta = 4t^2 + 3 \equiv 3 \pmod 4,
\]
which is not a square in $\QQ_2$. Hence $\phi(z) = \frac{1}{2} + t$ has no 
solution,  and $\QQ_2$ is not $\phi$-perfect. (It is worth noting that what we 
have  really proved is that the image of $\PP^1(\QQ_2)$ under $\phi$ is 
disjoint  from the set $\frac{1}{2} + \ZZ_2$.)

For $p =3$, we set $y = 2 + 3t$ with $t \in \ZZ_3$. Then $\ord_3(\Delta) = 1$,
which means $\Delta$ cannot be a square in $\QQ_3$.

Finally, we treat the case $p > 3$. The resultant of 
$\phi(z) = - \frac{z^2+2z}{2z+3}$ is $-3$, so this rational function may be 
reduced  modulo~$p$ to yield a quadratic function $\tilde \phi \in \FF_p(z)$. 
Note  that $\tilde \phi(0) = \tilde \phi(-2)$, so that $\tilde \phi$ fails to 
be injective  on $\PP^1(\FF_p)$. As $\PP^1(\FF_p)$ is a finite set, $\tilde \phi$ also 
fails  to be surjective. Choose $\tilde y \in \FF_p$ such that 
$\tilde \phi(z) = \tilde y$ has no solution in $\FF_p$, and choose a lift $y 
\in \ZZ_p$ such that $y \equiv \tilde y \pmod p$. It   
follows that $\phi(z) = y$ has no solution in $\ZZ_p$. It remains to show that 
$\phi(z) = y$ has no solution in $\QQ_p \smallsetminus \ZZ_p$. If $\phi(x) = y$ 
with $|x|_p > 1$, then 
\[
\left| \phi(x) \right|_p = |x|_p \cdot \left| \frac{1+2/x}{2 + 3/x} \right|_p = |x|_p > 1,
\]
which contradicts $y \in \ZZ_p$. Hence $\phi(z) = y$ has no solution in $\PP^1(\QQ_p)$, 
and we have proved that $\QQ_p$ is not $\phi$-perfect. 

%%%%%%%%%%%%%%%%%%%%%%%%%%%%%%%%%%%%%%%%%%%%%%%%%%%%%%%%%%%%%%%%%%%%%%%%%%%%%%%%
%%%%%%%%%%%%%%%%%%%%%%%%%%%%%%%%%%%%%%%%%%%%%%%%%%%%%%%%%%%%%%%%%%%%%%%%%%%%%%%%

\section{Further thoughts}

A general rational function of degree~$d > 2$ has $2d+1$ free parameters 
(coefficients)  and $2d-2$ critical points. Imposing the condition that 
$0,1,\infty$  are fixed and critical reduces this to $2d-5$ free parameters. If 
we  fix a set of $K$-rational critical points and look at the Wronskian, then the $2d-5$ 
free  coefficients for the function must satisfy $2d-5$ quadratic equations in 
$2d-5$  variables over $K$. In the case $d = 3$, in which $2d - 5 = 1$, we were 
able  to explicitly solve for the remaining critical point as an explicit 
function of  the free parameter. Is it possible to solve for the critical 
points as  explicit functions of the parameters for $d > 3$? 

B\'ezout's theorem gives an upper bound of $2^{2d-5}$ solutions for a general   
system of $2d-5$ conics, while Goldberg \cite{Goldberg_Catalan_1991} 
bounds the number of distinct solutions by the smaller quantity
\[
\frac{1}{d}\binom{2d-2}{d-1} \approx \frac{8}{\sqrt{\pi}d^{3/2}}2^{2d-5}. 
\]
This suggests a substantial amount of extra structure in our system of 
equations,  which may make it possible to give elementary proofs of the
theorem of Eremenko/Gabrielov in degree $d$ for other small $d>3$. 

\begin{comment}
\begin{question}
Does Theorem~\ref{Thm: Main theorem} extend to fields of characteristic~$p$? 
\end{question}

The proof and statement extend with little modification to fields of characteristic $p > 2$. However, the function $f_u$ degenerates to a quadratic when $p = 2$, essentially because a cubic is either totally ramified or wildly ramified at a critical point. If a cubic $f$ has totally ramified critical points $c_1, c_2 \in \PP^1(K)$, then it is equivalent to either $(z-c_1)^3$ or $\left(\frac{z-c_1}{z-c_2}\right)^3$, depending on whether $c_2 = \infty$. If it has precisely one critical point in $\PP^1(K)$, then we may assume it is $\infty$ without loss, so that it is equivalent to $az^2 + \frac{b}{cz}$. (Compare with \cite{Faber_One_Critical_Point_2014}.) 

Result is false when $p = 2$. Let $K$ be a non-algebraically closed subextension of $\bar \FF_2$. Then $f(z) = z^3 + az$ for $a \in \bar K \smallsetminus K$, but $\phi(z) = z^2$ is surjective on $\PP^1(K)$.  

\end{comment}

%%%%%%%%

\bibliographystyle{plain}
\bibliography{xander_bib}

\end{document}